\documentclass[12pt,english]{article}
\usepackage[T1]{fontenc}
\usepackage[latin9]{inputenc}
\usepackage{geometry}
\usepackage{babel}
\usepackage{amsmath}
\usepackage{amsthm}
\usepackage{amssymb}
\usepackage[numbers]{natbib}
\usepackage[unicode=true,pdfusetitle,
 bookmarks=true,bookmarksnumbered=false,bookmarksopen=false,
 breaklinks=false,pdfborder={0 0 1},backref=false,colorlinks=false]
 {hyperref}
\usepackage{appendix}
\usepackage{enumerate}

\makeatletter
\newcommand{\lyxaddress}[1]{
	\par {\raggedright #1
	\vspace{1.4em}
	\noindent\par}
}
\theoremstyle{plain}
\newtheorem{thm}{\protect\theoremname}[section]

\theoremstyle{plain}
\newtheorem{lemma}[thm]{\protect\lemmaname}

\theoremstyle{plain}
\newtheorem{prop}[thm]{\protect\propositionname}

\theoremstyle{definition}
\newtheorem{example}[thm]{\protect\examplename}

\theoremstyle{definition}
\newtheorem{rem}[thm]{\protect\remarkname}

\theoremstyle{plain}

\theoremstyle{plain}
\newtheorem*{prop*}{\protect\propositionname}

\theoremstyle{plain}
\newtheorem{cor}[thm]{\protect\corollaryname}

\theoremstyle{plain}

\@ifundefined{date}{}{\date{}}
\makeatother

\usepackage{listings}
\providecommand{\corollaryname}{Corollary}
\providecommand{\examplename}{Example}
\providecommand{\factname}{Fact}
\providecommand{\lemmaname}{Lemma}
\providecommand{\propositionname}{Proposition}
\providecommand{\remarkname}{Remark}
\providecommand{\theoremname}{Theorem}

\usepackage{fancyvrb}
\usepackage{tcolorbox}
\usepackage{color}
\definecolor{FireBrick}{rgb}{0.5812,0.0074,0.0083}
\definecolor{RoyalBlue}{rgb}{0.0236,0.0894,0.6179}
\definecolor{RoyalGreen}{rgb}{0.0236,0.6179,0.0894}
\definecolor{RoyalRed}{rgb}{0.6179,0.0236,0.0894}
\definecolor{LightBlue}{rgb}{0.8544,0.9511,1.0000}
\definecolor{Black}{rgb}{0.0,0.0,0.0}

\definecolor{promptColor}{rgb}{0.0,0.0,0.589}
\definecolor{brkpromptColor}{rgb}{0.589,0.0,0.0}
\definecolor{gapinputColor}{rgb}{0.589,0.0,0.0}
\definecolor{gapoutputColor}{rgb}{0.0,0.0,0.0}
\definecolor{DarkOlive}{rgb}{0.1047,0.2412,0.0064}

\tcbuselibrary{skins,breakable}
\newenvironment{BGVerbatim}[1]
{\VerbatimEnvironment
	\begin{tcolorbox}[enhanced,
		title=#1,
		fonttitle=\bfseries,
		attach boxed title to top left={xshift=-2mm,yshift=-2mm},
		boxed title style={colbacktitle=white},
		show bounding box,
		breakable,
		colback=gray!10,
		colframe=gray!30,
		spartan,
		titlerule=1mm,
		toprule=1mm,
		rightrule=0mm,
		bottomrule=1mm,
		leftrule=0mm,
		boxsep=-0.2mm
		]\begin{Verbatim}[commandchars=!@|,fontsize=\small,frame=lines,framerule=0mm, rulecolor=\color{gray!60}]}
		{\end{Verbatim}\end{tcolorbox}}

\begin{document}
\title{Bounds for regular induced subgraphs of strongly regular graphs}
\author{Rhys J. Evans}

\maketitle
\vspace{-3em}

\lyxaddress{\begin{center}
Sobolev Institute of Mathematics\par
4 Acad. Koptyug avenue\par
630090 Novosibirsk, Russia\par
\tt{rhysjevans00@gmail.com}\par
\vspace{1em} 
\normalfont \today
\par\end{center}}

\begin{abstract}
Given feasible strongly regular graph parameters $(v,k,\lambda,\mu)$  and a non-negative integer $d$, we determine upper and lower bounds on the order of a $d$-regular induced subgraph of any strongly regular graph with parameters $(v,k,\lambda,\mu)$. Our new bounds are at least as good as the bounds on the order of a $d$-regular induced subgraph of a $k$-regular graph determined by Haemers. Further, we prove that for each non-negative integer $d$, our new upper bound improves on Haemers' upper bound for infinitely many strongly regular graphs.
\end{abstract}

\section{Introduction}

The question of finding the maximum order of a $d$-regular induced subgraph of a given
graph $\Gamma$ is a generalisation of many problems in graph theory. Some examples of these include finding the independence number, clique number and the order of a maximum induced matching in a given graph. In general, finding a $d$-regular induced subgraph of a given graph $\Gamma$ is computationally hard (see Asahiro et al. \cite{AEIM_2014}). Significant improvements in computational time can be made by using bounds on the order of a $d$-regular induced subgraph of $\Gamma$ to reduce the search space of the problem.

Haemers \cite{H_1979} gives an upper and lower bounds on the order of a $d$-regular induced subgraph of a $v$-vertex $k$-regular graph with given least and second largest eigenvalues, which generalises an unpublished result of Hoffman (see \cite{H_2021}). More recently, Cardoso, Karminski and Lozin \cite{CKL_2007} derive the same upper bound as a consequence of semidefinite programming methods which can be applied to any graph.
Considering a strongly regular graph $\Gamma$ with parameters $(v,k,\lambda,\mu)$, Neumaier \cite{N_1982} derives the same upper and lower bounds on the order of a $d$-regular induced subgraph of $\Gamma$, through the use of a combinatorial
argument. 

Greaves and Soicher \cite{GS_2016} analyse an upper bound on the order of cliques in an edge-regular graph with given parameters, called the clique adjacency bound. They prove that given any strongly regular graph $\Gamma$, the clique adjacency bound is at least as good as the well-known Delsarte bound \cite{D_1975}. Furthermore, they find infinitely many strongly regular graphs for which the clique adjacency bound is strictly better than the Delsarte bound. Greaves et al. \cite{GKP_2020} also improve on the Delsarte bound when the parameters of a strongly regular graph meet certain conditions. 

In this paper, we generalise certain results of Greaves and Soicher \cite{GS_2016}, where instead of cliques, we will consider $d$-regular induced subgraphs. In Section \ref{sec:prelim} we introduce known results on strongly regular graphs and their spectra. In Section \ref{sec:specb} we present bounds on regular induced subgraphs of regular graphs given by Haemers \cite{H_1979}. In Section \ref{sec:rabs}, we use the block intersection polynomials defined in Soicher \cite{S_2010} to determine upper and lower bounds on the order of a $d$-regular induced subgraph of any strongly regular graph with parameters $(v,k,\lambda,\mu)$. These bounds are also introduced in Brouwer and Van Maldeghem \cite[Section 1.1.14]{BV_2022}. In Section \ref{sec:compb}, we show that our new bounds are at least as good as the bounds on the order of a $d$-regular induced subgraph of a $k$-regular graph determined by Haemers \cite{H_1979}. 

In Section \ref{sec:strict}, we analyse our upper bound for type I and type II strongly regular graphs separately. Consequently, we prove that for each non-negative integer $d$, our upper bound improves on Haemers' upper bound for infinitely many type I and infinitely many type II strongly regular graphs. At the end of this section, we carry out computations using the \textsf{AGT} package \cite{AGT_2020} to verify the new bounds beat Haemers' bounds relatively often for strongly regular graphs of small order.

In Section \ref{sec:cab} we comment on the relationship between our bounds with the clique adjacency bound (CAB) of Soicher \cite{S_2015}. For strongly regular graphs, we see that the CAB is equivalent to one of our new bounds. In fact for any edge-regular graph, the CAB is at least as good as Hoffman's ratio bound of the complement of the graph (see Haemers \cite{H_2021} for more on Hoffman's bound). We then use the \textsf{AGT} package \cite{AGT_2020} to analyse how good the CAB is for small strongly regular graphs. 

In appendix \ref{app:maple} we use the Groebner package in MAPLE \cite{MAPLE_2011} to verify calculations made in Sections \ref{sec:rabs}, \ref{sec:compb}, and \ref{sec:strict}.

\section{Preliminaries}\label{sec:prelim}

A \emph{graph} is an ordered pair $\Gamma=(V,E)$, where $V$ is a finite set and $E$ is a set of subsets of size $2$ of $V$. Then, the members of $V$ are called the \emph{vertices} of $\Gamma$, and the members of $E$ are called the \emph{edges} of $\Gamma$. We denote the set of vertices of the graph $\Gamma$ by $V(\Gamma)$, and the set of edges of $\Gamma$ by $E(\Gamma)$.

Now let $\Gamma$ be a graph. The \emph{order} of $\Gamma$ is the cardinality $|V(\Gamma)|$ of its vertex set. For any two distinct vertices $u,w$ of $\Gamma$, we denote by $uw$ the set $\{u,w\}$, and $u,w$ are said to be \emph{adjacent} if $uw\in E(\Gamma)$. We do not consider a vertex to be adjacent to itself. Let $u$ be a vertex of $\Gamma$. The \emph{neighbourhood} of $u$ is the set of vertices adjacent to $u$, and is denoted by $\Gamma(u)$. The \emph{degree} of $u$ is the cardinality $|\Gamma(u)|$ of its neighbourhood.  

Consider a set of vertices $U\subseteq V(\Gamma)$. The \emph{induced subgraph} of $\Gamma$ on $U$, denoted by $\Gamma[U]$, is the graph with vertex set $U$, and vertices in $\Gamma[U]$ are adjacent if and only if they are adjacent in $\Gamma$.

Let $v$ be the order of $\Gamma$. The \emph{adjacency matrix} of $\Gamma$,  $A(\Gamma)$, is the $v\times v$ matrix indexed by $V(\Gamma)$ such that $A(\Gamma)_{xy}=1$ if $xy\in E(\Gamma)$, and $A(\Gamma)_{xy}=0$ otherwise. An \emph{eigenvalue} of $\Gamma$ is an eigenvalue of the matrix $A(\Gamma)$.

A graph $\Gamma$ is \emph{$k$-regular} if every vertex of $\Gamma$ has degree $k$.
A graph $\Gamma$ is \emph{strongly regular} with \emph{parameters} $(v,k,\lambda,\mu)$ if $\Gamma$ non-complete, non-null, every pair of adjacent vertices have exactly $\lambda$ common neighbours, and every pair of distinct nonadjacent vertices have exactly $\mu$ common neighbours. We denote by $\text{SRG}(v,k,\lambda,\mu)$ the set of strongly regular graphs with parameters $(v,k,\lambda,\mu)$. 

For a strongly regular graph $\Gamma$ with parameters $(v,k,\lambda,\mu)$, it is known that $\Gamma$ has at most $3$ distinct eigenvalues, with largest eigenvalue $k$ (see Brouwer and Haemers \cite[Theorem 9.1.2]{BH_2011}). The \emph{restricted eigenvalues} of a strongly regular graph are the eigenvalues of the graph with eigenspaces perpendicular to the all-ones vector. We often denote these eigenvalues by $\rho,\sigma$, with $k\geq \rho>\sigma$. The following shows that the eigenvalues of strongly regular graphs only depend on the parameters of the graph.

\begin{prop}
	\label{prop:params} 
	Let $\Gamma$ be in $\text{SRG}(v,k,\lambda,\mu)$
	and $\rho>\sigma$ be the restricted eigenvalues of $\Gamma$. Then
	\begin{enumerate}
		\item $\rho$ and $\sigma$ are uniquely determined from the parameters $(v,k,\lambda,\mu)$, we have $\rho\geq0,\sigma<0$,  and the following relations hold. 
		\begin{eqnarray*}\mu(v-k-1) & = & k(k-\lambda-1)\\
			\lambda-\mu & = & \rho+\sigma\\
			\mu-k & = & \rho \sigma
		\end{eqnarray*}
		\item If $\rho,\sigma$ are not integers, then there exists a positive integer $n$ such that $(v,k,\lambda,\mu)=(4n+1,2n,n-1,n)$.
	\end{enumerate}	
\end{prop}
\begin{proof}
	This is a routine calculation that uses the properties of the adjacency matrix of a strongly regular graph, and can be found in Brouwer and Haemers \cite[Theorem 9.1.3]{BH_2011}.
\end{proof}
Using this Proposition, we can derive the following useful identity. This is a well-known identity, and can be found in Brouwer, Cohen and Neumaier \cite[Theorem 1.3.1(iii)]{BCN_1989}.
\begin{lemma}
	\label{prop:vmu}Let $\Gamma$ be in $\text{SRG}(v,k,\lambda,\mu)$
	with restricted eigenvalues $\rho>\sigma$. Then $v\mu=(k-\sigma)(k-\rho)$.
\end{lemma}
\begin{proof}
	This follows immediately from Proposition \ref{prop:params}.
\end{proof}

Proposition \ref{prop:params} 2. describes the division of strongly regular graphs into two main classes, called type I and type II. A strongly regular graph $\Gamma$ is of \emph{type I}, or a \emph{conference graph}, if $\Gamma$ is in $\text{SRG}(4n+1,2n,n-1,n)$
for some positive integer $n$. A strongly regular graph $\Gamma$ is of \emph{type II} if all eigenvalues of $\Gamma$ are integer. These two classes of strongly regular graphs are not mutually exclusive (for more about this, see \cite{BH_2011}).

\section{Haemers' spectral bounds}\label{sec:specb}

In his thesis, Haemers \cite{H_1979} derives bounds on the order of induced subgraphs of regular graphs using eigenvalue techniques. 

\begin{prop}\label{upeig}
	Let $\Gamma$ be a $k$-regular graph of order $v$ with smallest eigenvalue $\sigma$. Suppose $\Gamma$ has an induced subgraph $\Delta$ of order $y>0$, and average vertex-degree $d$. Then 
	\begin{equation*}
	y\leq v\left(\frac{d-\sigma}{k-\sigma}\right).
	\end{equation*}
\end{prop} 
\begin{proof}
	This is a standard result that uses eigenvalue interlacing. The proof can be found in Haemers \cite[Theorem 2.1.4]{H_1979}.
\end{proof}
A lower bound can also be derived when considering connected graphs.
Note that this bound need not be positive.
\begin{prop}
	Let $\Gamma$ be a $k$-regular graph of order $v$ with second largest eigenvalue $\rho$. Suppose $\Gamma$ has an induced subgraph $\Delta$ of order $y>0$, and average vertex-degree $d$. Then 
	\begin{equation*}
	y\geq v\left(\frac{d-\rho}{k-\rho}\right).
	\end{equation*}
\end{prop}
\begin{proof}
	This is a standard result that uses eigenvalue interlacing. The proof can be found in Haemers \cite[Theorem 2.1.4]{H_1979}.
\end{proof}
For a $k$-regular graph $\Gamma$ of order $v$ and with smallest eigenvalue $\sigma$,  we define the upper bound of Haemers 
\begin{equation}
\text{Haem}_{\geq}(\Gamma,d):=v\left(\frac{d-\sigma}{k-\sigma}\right)
\end{equation}
and if $\Gamma$ is connected with second largest eigenvalue $\rho$, we define the lower bound of Haemers
\begin{equation}
\text{Haem}_{\leq}(\Gamma,d):=v\left(\frac{d-\rho}{k-\rho}\right).
\end{equation}
We note that $\text{Haem}_{\geq}(\Gamma,0)$ coincides with the well-known Hoffman ratio bound \cite{H_2021}, so the bounds of Haemers' generalise the Hoffman ratio bound.

\begin{example}\label{exmp:slg}
	In this example, we will see that for certain cases of (strongly) regular graphs, the upper and lower bounds of Haemers are attained.
	
	For $n\geq2$, the \emph{square lattice graph} $L_{2}(n)$
	has vertex set $\{1,2,...,n\}\times\{1,2,...,n\} $, and two distinct vertices are joined
	by an edge precisely when they have the same value at one coordinate.
	This graph is strongly regular with parameters $(n^{2},2(n-1),n-2,2)$, and has eigenvalues $k=2n-2,\rho=n-2,\sigma=-2$.

	Now consider the induced subgraph $\Delta$ with vertex set consisting of the complement of two distinct columns. Then $\Delta$ is a $(2n-4)$-regular induced subgraph of order $n^{2}-2n$, which is the lower bound of Haemers.
	\begin{figure}[ht!]
		\centering
		\includegraphics[width=50mm]{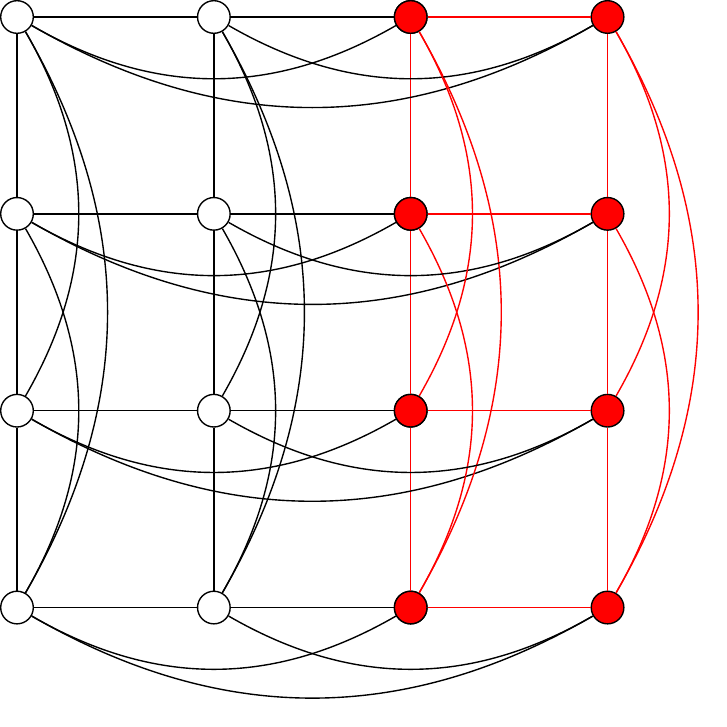} 
		\caption[Haemers' lower bound]{A regular induced subgraph attaining Haemers' lower bound}
		\label{fig:hlb}
	\end{figure} 
	
	Now consider the induced subgraph $\Delta$ with vertex set consisting of the complement of a maximum-size independent set. Then $\Delta$ is a $(2n-4)$-regular induced subgraph of order $n^{2}-n$, which is the upper bound of Haemers.

	\begin{figure}[ht!]
		\centering
		\includegraphics[width=50mm]{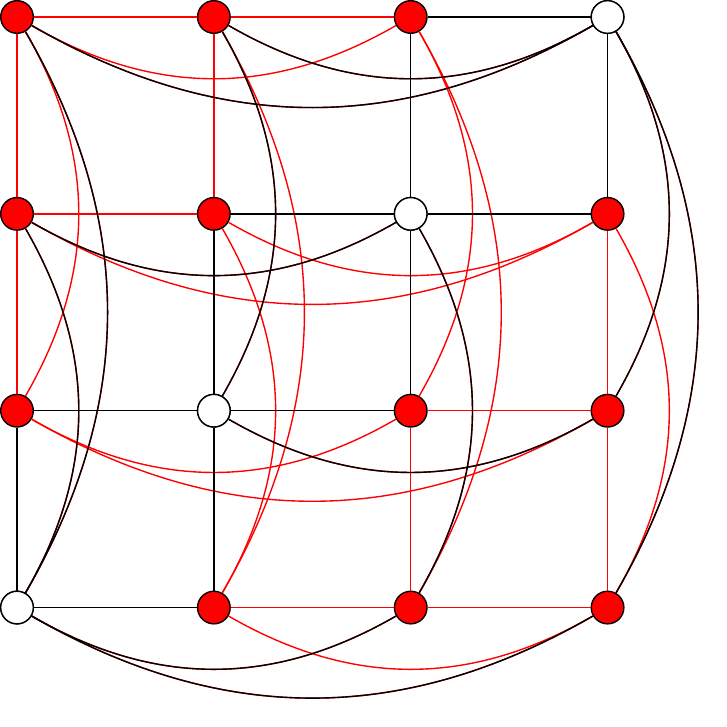} 
		\caption[Haemers' upper bound]{A regular induced subgraph attaining Haemers' upper bound}
		\label{fig:hub}
	\end{figure} 
\end{example}	
\pagebreak

\section{The Regular Adjacency Bounds}\label{sec:rabs}

Our aim for this section is to derive bounds on the order of a $d$-regular induced subgraph of a strongly regular graph, using a certain block intersection polynomial. This approach generalises results of Greaves and Soicher \cite{GS_2016}.

Let $\Gamma\in \text{SRG}(v,k,\lambda,\mu)$. We define the \emph{regular adjacency polynomial} for the graph $\Gamma$, or \emph{rap},
as the polynomial
\begin{equation*}
R_{\Gamma}(x,y,d):=x(x+1)(v-y)-2xyk+(2x+\lambda-\mu+1)yd+y(y-1)\mu-yd^{2}.
\end{equation*}
This is the polynomial found in Soicher \cite[Theorem 1.2]{S_2015},
applied with constant degree sequence $(d,d,...,d)$. Note that we
are dealing with strongly regular graphs, so we do not need to consider
the diameter condition stated in the theorem. 

This polynomial has some useful properties, which come from the fact that it is a block intersection polynomial.

\begin{thm}
	\label{thm:rapcon}Let $\Gamma$ be in $\text{SRG}(v,k,\lambda,\mu)$
	and $\Delta$ be a $d$-regular induced subgraph of order $y\geq 2$ in
	$\Gamma$. Then $R_{\Gamma}(m,y,d)\geq0$ for all integers $m$. 
\end{thm}
\begin{proof}
	This result is an application of Soicher \cite[Theorem 1.1]{S_2010}.
\end{proof}
Note that given a set $S\subseteq V(\Gamma)$ such that $\Gamma[S]$ is $d$-regular, it is not necessarily true that a proper subset of $S$ induces a $d$-regular subgraph. Because of this, how we define bounds from the properties of the regular adjacency polynomial will be slightly different to how the clique adjacency bound is defined in Soicher \cite{S_2015} (which is derived from the properties of another block intersection polynomial, the clique adjacency polynomial).

Consider the set 
\begin{equation*}
S_{d}:=\{y\in \{d+1,\dots,v\}:\text{ for all integers }x, R_{\Gamma}(x,y,d)\geq 0  \}.
\end{equation*}
We define the \emph{regular adjacency upper bound}, or \emph{raub} of the strongly regular graph $\Gamma$ as
\[
\text{Rab}_{\geq}(\Gamma,d):=\begin{cases}
\max(S_{d}) & S_{d}\not=\emptyset,\\
0 & \text{otherwise,}
\end{cases}
\]
and we define the \emph{regular adjacency lower bound}, or \emph{ralb} of the strongly regular graph $\Gamma$ as
\[
\text{Rab}_{\leq}(\Gamma,d):=\begin{cases}
\min(S_{d}) & S_{d}\not=\emptyset\\
v+1 & \text{otherwise.}
\end{cases}
\]
Note that these bounds are the same for any two distinct graphs in $\text{SRG}(v,k,\lambda,\mu)$. 

After dealing with a trivial case, we can now use Theorem \ref{thm:rapcon} to prove that the graph $\Gamma$  has no non-empty $d$-regular induced subgraph of order greater than $\text{Rab}_{\geq}(\Gamma,d)$ or less than $\text{Rab}_{\leq}(\Gamma,d)$. 

\begin{thm}
	\label{thm:rabs}Let $\Gamma$ be in $\text{SRG}(v,k,\lambda,\mu)$ and $\Delta$ be a $d$-regular induced subgraph of order $y>0$ in
	$\Gamma$. Then 
	\begin{equation*}
	\text{Rab}_{\leq}(\Gamma,d)\leq y\leq\text{Rab}_{\geq}(\Gamma,d).
	\end{equation*}
\end{thm}
\begin{proof}
	It is easy to see that $\Delta$ has to have at least $d+1$ vertices. If $y\geq 2$, then by Theorem \ref{thm:rapcon}, we have $R_{\Gamma}(x,y,d)\geq 0$ for all integers $x$. By the definitions of the raub and ralb,
	\begin{equation*}
	\text{Rab}_{\leq}(\Gamma,d)\leq y\leq\text{Rab}_{\geq}(\Gamma,d).
	\end{equation*}
	The only case left to consider is when $y=1$. As $\Delta$ has at least $d+1$ vertices, we must have $d=0$. Consider
	\vspace{-.25em}
	\begin{equation*}
	R_{\Gamma}(x,1,0)=x\left((x+1)(v-1)-2k\right).
	\end{equation*}
	\vspace{-.25em}
	This polynomial in $x$ has roots $x_{1}=0$ and $x_{2}=2k/(v-1)-1$. As $k\leq v-1$, we have $-1 \leq x_{2}\leq 1$.
	
	Therefore, we have the following three cases.
	\vspace{-.25em}
	\begin{enumerate}
		\item  $-1\leq x_{2} < 0$ and $R(x,1,0)$ is negative only for $x$ lying in an open interval contained in $(-1,0)$. 
		\vspace{-.3em}
		\item $x_{2} = 0$ and $R(x,1,0)$ is non-negative for all $x$.
		\vspace{-.3em}
		\item $0 < x_{2} \leq 1$ and $R(x,1,0)$ is negative only for $x$ lying in an open interval contained in $(0,1)$. 
	\end{enumerate}
	\vspace{-.25em} 		
	In each case, for all integers $x$ we have $R(x,1,0)\geq 0$. By definition, we see that
	\begin{equation*}
	\text{Rab}_{\leq}(\Gamma,0)\leq 1\leq\text{Rab}_{\geq}(\Gamma,0).
	\end{equation*}	
\end{proof}

\section{Comparison of bounds}\label{sec:compb}

We will now compare the bounds of Haemers from Section \ref{sec:specb} with the raub and ralb defined in Section \ref{sec:rabs}. For $\Gamma\in \text{SRG}(v,k,\lambda,\mu)$, and non-negative integer $d\leq k$, the next three propositions show that $\text{Rab}_{\geq}(\Gamma,d)\leq\lfloor\text{Haem}_{\geq}(\Gamma,d)\rfloor$,
and $\text{Rab}_{\leq}(\Gamma,d)\geq\lceil\text{Haem}_{\leq}(\Gamma,d)\rceil$.

First we note that for any strongly regular graph parameters $(v,k,\lambda,\mu)$, each of the bounds $$\text{Haem}_{\geq}(\Gamma,d),\text{Haem}_{\leq}(\Gamma,d),\text{Rab}_{\geq}(\Gamma,d),\text{Rab}_{\leq}(\Gamma,d)$$ on $d$-regular induced subgraphs is independent of the choice of the graph $\Gamma$ in $\text{SRG}(v,k,\lambda,\mu)$. Therefore, we will only be concerned with fixed parameter sets and their corresponding restricted eigenvalues.

In the following we present a useful value, which will be important throughout the remainder of this paper.

\begin{rem}\label{rem:x_y}
	At most values of $y$ and $d$, the polynomial
	$R_{\Gamma}(x,y,d)$ is a quadratic in $x$ with positive leading coefficient. For $\Gamma\in \text{SRG}(v,k,\lambda,\mu)$, non-negative integer $d\leq k$ and $y$ in the range $0<y<v$, $R_{\Gamma}(x,y,d)$ is a polynomial in $x$ which it attains its minimum value at 
	\begin{equation}\label{eq:x_y}
		x_y=\frac{2y(k-d)-(v-y)}{2(v-y)}.
	\end{equation}
\end{rem}

We will now use this value to prove that the regular adjacency bounds are at least as good as Haemers' bounds by the following observations. For any fixed value of $y$ in the ranges $0<y<\text{Haem}_{\leq}(\Gamma,d)$ and $\text{Haem}_{\geq}(\Gamma,d)<y<v$,
we will see that the quadratic $R_{\Gamma}$ in $x$ is negative on an open interval
of length strictly greater than $1$. Every interval of length more
than $1$ must contain an integer, with which we can then use in applying
Theorem \ref{thm:rapcon}.

\begin{lemma}\label{lem:upbound}
	Let $\Gamma$ be in $\text{SRG}(v,k,\lambda,\mu)$ where $\mu\not=0$ and $d$ be an integer, $0\leq d \leq k$. For all $y$ such that $0<y<\text{Haem}_{\leq}(\Gamma,d)$
	or $\text{Haem}_{\geq}(\Gamma,d)<y<v$, there is an integer $b_{y}$ such that $R_{\Gamma}(b_{y},y,d)<0$. 
\end{lemma}
\begin{proof}
	Let $0<y<v$. Then $R_{\Gamma}(x,y,d)$ is a quadratic polynomial with
	positive leading coefficient. In Remark \ref{rem:x_y} we note that $R_{\Gamma}(x,y,d)$ attains its minimum value in $x$ at the point $x_{y}$ (Equation \eqref{eq:x_y}). If $R_{\Gamma}(x_{y}+1/2,y,d)<0$,
	by symmetry of the quadratic around $x=x_{y}$, we must have $R_{\Gamma}(x_{y}-1/2,y)<0$,
	and so we have $R_{\Gamma}(x,y,d)<0$ for all  $x\in[x_{y}-1/2,x_{y}+1/2]$.
	This is an interval of size $1$, so must contain an integer $b_{y}$,
	and $R_{\Gamma}(b_{y},y,d)<0$. We claim $R(x_{y}+1/2,y,d)<0$ for
	$y>\text{Haem}_{\geq}(\Gamma,d)$ or $y<\text{Haem}_{\leq}(\Gamma,d)$, which proves the result.
	
	Using equation \eqref{eq:x_y}, we see that
	\begin{equation*}
		x_{y}+\frac{1}{2} = \frac{(k-d)y}{v-y}
	\end{equation*}
	Let $\rho>\sigma$ be the restricted
	eigenvalues corresponding to the strongly regular graphs parameters
	$(v,k,\lambda,\mu)$. We then establish the following identity, using the relations for
	strongly regular graph parameters (and is verified using Maple in Appendix \ref{app:maple}).
	\begin{eqnarray*}
		-\frac{(v-y)}{y}R_{\Gamma}(x_{y}+1/2,y,d) & = & \mu y^{2}-((d-\rho)(k-\sigma)+(d-\sigma)(k-\rho))y\\
		&  & \qquad+v(d-\sigma)(d-\rho).
	\end{eqnarray*}
	Then multiply by $\mu$ and deduce the following identity by using Lemma \ref{prop:vmu}. 
	\begin{equation}\label{eq:rapfact}
	-\frac{(v-y)}{y}\mu R_{\Gamma}(x_{y}+1/2,y,d) = (\mu y-(d-\rho)(k-\sigma))(\mu y-(d-\sigma)(k-\rho))
	\end{equation}
	(this is also verified using Maple in Appendix \ref{app:maple}). Consider the right side of Equation \eqref{eq:rapfact} as a quadratic in $y$. Take the roots
	of this quadratic, 
	\begin{equation*}
	\alpha=\frac{(d-\rho)(k-\sigma)}{\mu},\beta=\frac{(d-\sigma)(k-\rho)}{\mu}.
	\end{equation*}
	As $(k-\rho)(d-\sigma)-(k-\sigma)(d-\rho)=(\rho-\sigma)(k-d)$ is positive, we have $\beta\geq\alpha$.
	By Lemma \ref{prop:vmu}, $\beta=\text{Haem}_{\geq}(\Gamma,d)$ and $\alpha=\text{Haem}_{\leq}(\Gamma,d)$. We know that $\mu(v-y)/y>0$, so we have $R_{\Gamma}(x_{y}+1/2,y,d)<0$
	if and only if the right side of Equation \eqref{eq:rapfact} is positive. This is exactly when
	$y>\beta=\text{Haem}_{\geq}(\Gamma,d)$ or $y<\alpha=\text{Haem}_{\leq}(\Gamma,d)$. 
\end{proof}
Now we deal with the case when $\mu=0$.
\begin{lemma}\label{lem:upboundmu0}
	Let $\Gamma$ be in $\text{SRG}(v,k,\lambda,\mu)$ where $\mu=0$ and $d$ be an integer, $0\leq d \leq k$. For all $y$ such that $\text{Haem}_{\geq}(\Gamma,d)<y<v$, there is an
	integer $b_{y}$ such that $R_{\Gamma}(b_{y},y,d)<0$.
\end{lemma}
\begin{proof}
	For $\mu=0$, we have $k=\rho,\sigma=-1$. Using the same notation and approach as Lemma \ref{lem:upbound} and using Lemma \ref{prop:vmu}, we see that
	\begin{equation*}
	-\frac{(v-y)}{y}R_{\Gamma}(x_{y}+1/2,y,d)=(k-d)(k+1)y-v(k-d)(d+1).
	\end{equation*}
	The right side is strictly
	greater than $0$ for $\text{Haem}_{\geq}(\Gamma,d)<y<v$. Therefore $R(x_{y}+1/2,y,d)<0$
	for all such $y$. Applying a similar argument to Lemma \ref{lem:upbound},
	we are done.
\end{proof}

Finally, we deal with the case when $y=v$. We would like our bound
to allow for a regular subgraph of order $v$ if and only if $d=k$,
as this is the degree of the only regular subgraph of order $v$.
The following shows that this is true.
\begin{lemma}
	\label{lem:vbound}	Let $\Gamma$ be in $\text{SRG}(v,k,\lambda,\mu)$ and $d$ be an integer, $0\leq d \leq k$. Then there is an integer $b_{v}$ such that $R_{\Gamma}(b_{v},v,d)<0$
	if and only if $d\not =k$.
\end{lemma}
\begin{proof}
	In this case, $R_{\Gamma}(x,v,d)$ is a linear function in $x$. If $k\not=d$,
	$R_{\Gamma}(x,y,d)$ is non-constant, so trivially there is such a $b_{v}$.
	Otherwise 
	\[
	R_{\Gamma}(x,v,d)=v((\rho+\sigma+1)k+v\mu-\mu-k^{2})
	\]
	which is $0$ by Lemma \ref{prop:vmu} (and verified by using Maple in Appendix \ref{app:maple}).
\end{proof}
With the above results we have covered all possible cases needed to prove the following theorem.
\begin{thm}\label{thm:rabcomp}
	Let $\Gamma$ be in $\text{SRG}(v,k,\lambda,\mu)$ and $d$ be an
	integer $0\leq d \leq k$. Then $\text{Rab}_{\geq}(\Gamma,d)\leq\lfloor\text{Haem}_{\geq}(\Gamma,d)\rfloor$
	and $\text{Rab}_{\leq}(\Gamma,d)\geq\lceil\text{Haem}_{\leq}(\Gamma,d)\rceil$. 
\end{thm}
\begin{proof}
	Note that $\text{Rab}_{\geq}(\Gamma,d),\text{Rab}_{\leq}(\Gamma,d)$ are integers by definition, so we only need
	to show $\text{Rab}_{\geq}(\Gamma,d)\leq \text{Haem}_{\geq}(\Gamma,d)$ and $\text{Rab}_{\leq}(\Gamma,d)\leq \text{Haem}_{\leq}(\Gamma,d)$.
	
	Take any integer $i$ such that $\text{Haem}_{\geq}(\Gamma,d)<i<v$, or $0<i<\text{Haem}_{\leq}(\Gamma,d)$ if $\text{Haem}_{\leq}(\Gamma,d)>0$.
	Then by Lemmas \ref{lem:upbound}, \ref{lem:upboundmu0} and \ref{lem:vbound}, there is a  $b_{i}\in\mathbb{Z}$ such that $R_{\Gamma}(b_{i},i,d)<0$.
	Thus by Theorem \ref{thm:rabs} and the definitions of $\text{Rab}_{\geq}(\Gamma,d)$
	and $\text{Rab}_{\leq}(\Gamma,d)$, the result follows.
\end{proof}

\section{Improving on Haemers' upper bound}\label{sec:strict}

In this section, we consider when the regular adjacency upper bound is strictly less than Haemers' upper bound. Our approach will be to consider Type I and Type II strongly regular graphs separately.

First, let us reintroduce some notation. For any real number $x$, we define 
\begin{equation*}
  \left[x\right]:=\lceil x-1/2 \rceil
\end{equation*}
In other words, $\left[x\right]$ is the smallest nearest integer to $x$. We also define the \emph{fractional part} of a real number $x$ as
\begin{equation*}
\text{frac}(x):=x-\lfloor x\rfloor.
\end{equation*}

Let $\Gamma\in \text{SRG}(v,k,\lambda,\mu)$ with restricted eigenvalues $\rho>\sigma$, and $d$ be a non-negative integer where $d< k$. In the proof of Lemma \ref{lem:upbound} we use the value $x_y$ from Remark \ref{rem:x_y}. In particuar for $y<v$, the critical point of the quadratic  $R_{\Gamma}(x,y,d)$ in $x$ is $x_y$. Therefore 
\begin{equation}\label{eq:R_gefficient}
   R_{\Gamma}(b_y,y,d)<0 \quad \text{for some integer }b_y \iff R_{\Gamma}(\left[x_y\right],y,d)<0.
\end{equation} 
In fact we can use continuity arguments to see for $\text{Haem}_{\leq}(\Gamma,d)<y<\text{Haem}_{\geq}(\Gamma,d)$, the only possible integer $b_y$ for which $R_{\Gamma}(b_y,y,d)<0$ is $\left[x_y\right]$.

Let $h_d=\text{Haem}_{\geq}(\Gamma,d)$. To prove $Rab_{\geq}(\Gamma,d)<\lfloor h_d \rfloor$, it suffices to show that $R_{\Gamma}(\left[ x_{\lfloor h_d \rfloor}\right],\lfloor h_d \rfloor,d)<0$. It would be useful if we knew that $\left[ x_{\lfloor h_d \rfloor}\right]=\left[x_{h_d}\right]$, as we already know the value $x_{h_d}$. Although this is not necessarily true, we focus on this case in our analysis in order to derive sufficient conditions for the raub to better better than Haemers' upper bound.

\subsection{Type I strongly regular graphs}\label{sec:tIstrict}

In this section we derive some sufficient conditions for the raub to be strictly better than Haemers' upper bound for a type I strongly regular graph which is also not of type II (its restricted eigenvalues are non-integers). In this case, we will have to deal with two unknown fractional parts, coming from $h_d$ and $x_{h_d}$.

The parameters and eigenvalues of a type I strongly regular graph can be represented as follows.

\begin{lemma}\label{prop:t1eig}
	Let $\Gamma$ be in $\text{SRG}(v,k,\lambda,\mu)$ and of type
	I. Then $\Gamma$ has eigenvalues $k$, $\rho=(\sqrt{v}-1)/2$ and $\sigma=(-\sqrt{v}-1)/2$. Furthermore, we have
	\begin{eqnarray*}
		v&=&(2\sigma +1)^2,\\
		k&=&2\sigma(\sigma +1),\\
		\lambda&=&\sigma(\sigma+1)-1,\\
		\mu&=&\sigma(\sigma+1).  
	\end{eqnarray*} 
\end{lemma}
\begin{proof}
	This follows from definition of type I graphs and Proposition \ref{prop:params}.
\end{proof}

Throughout this section, $\Gamma$ will be in $\text{SRG}(4n+1,2n,n-1,n)$ for some positive integer $n$, and $d$ is assumed to be a non-negative integer where $d\leq 2n$. Let $\rho,\sigma$ be the restricted eigenvalues of $\Gamma$, which can be expressed in terms of $n$ by using Proposition \ref{prop:t1eig}. Let us assume that $\Gamma$ is not of type II as well. First we see what the value of $\text{Haem}_{\geq}(\Gamma,d)$ is in the case of type I graphs.
\begin{prop}
	\label{prop:t1bound}Let $\Gamma$ be in $\text{SRG}(v,k,\lambda,\mu)$
	and of type I, with restricted eigenvalues $\rho>\sigma$. Then
	\begin{equation*} \text{Haem}_{\geq}(\Gamma,d)=2d-2\sigma+\frac{d}{\sigma}-1.
	\end{equation*}
\end{prop}
\begin{proof}
	This follows from the identities in Proposition \ref{prop:t1eig}.
\end{proof}

Let $h_d=\text{Haem}_{\geq}(\Gamma,d)$ and $t=\text{frac}(-\sigma)$. Then $x_{h_d}=d-\sigma-1/2$, and as $\Gamma$ is not of type II, 
\begin{equation}
\left[x_{h_d}\right]=d-\sigma-t.
\end{equation}
Let $f=\text{frac}(h_d)$. For us to analyse the regular adjacency polynomial around $(x_{h_d},h_d,d)$, we would like to relate $f$ to $t$ in some way. When $d<-\sigma$ we can take cases on $\text{frac}(-2\sigma)$ to get the following.
\begin{equation}
f=\begin{cases}
2t+d/\sigma+1 \qquad &t<-d/2\sigma\\
2t+d/\sigma \qquad &-d/2\sigma<t<1/2-d/2\sigma\\
2t+d/\sigma - 1 \qquad &t>1/2-d/2\sigma
\end{cases}
\end{equation}

In the following results, it will be useful to remember that $f=2t+d/\sigma +a -1$, where $a=0,1$ or $2$. Now we will analyse the rap around the point $$\left(\left[x_{h_d}\right],\lfloor h_d\rfloor,d\right)=(d-\sigma-t,{h_d}-f,d).$$ In this case, it is also useful to observe that $\lfloor h_d \rfloor=2(d-\sigma-t)-a$.
\begin{lemma}
	\label{lem:t1uppoly}Let $\Gamma$  be in $\text{SRG}(v,k,\lambda,\mu)$
	and of type I with restricted eigenvalues $\rho>\sigma$. Then 
	\begin{equation}\label{eq:t1quad}
	\begin{aligned}
	R_{\Gamma}(d-\sigma-t,&2(d-\sigma-t)-a,d)=\\&-(d-\sigma-t)(2t^{2}-(1-4\sigma)t+d-3\sigma-1)\\
	&+a(t^2 + (2\sigma-1) t+2\sigma^2+d+a\sigma(\sigma+1)).
	\end{aligned}
	\end{equation}
\end{lemma}
\begin{proof}
	This can be proven by manipulating the expression $R_{\Gamma}(z,2z-a)$, where $z=(d-\sigma-t)$, into two parts, one containing only parts not divisible by $a$.
\end{proof}
For the cases when $a=0$, we can solve for the roots of this polynomial in $t$. From this, deduce the following.
\begin{prop}
	\label{prop:strictgen}Let $\Gamma$ be in $\text{SRG}(v,k,\lambda,\mu)$
	and of type I with restricted eigenvalues $\rho>\sigma$, and $d$ be a non-negative integer where $d< -\sigma$. If 
	\[
	\frac{1}{2}+\frac{d}{\sqrt{v}+1}<frac(-\sigma)<\frac{3}{4}+\frac{\sqrt{v}-\sqrt{v-2d+5/4}}{2},
	\]
	then $\text{Rab}_{\geq}(\Gamma,d)< \lfloor\text{Haem}_{\geq}(\Gamma,d) \rfloor$.
\end{prop}
\begin{proof}
	Let $h_d=\text{Haem}_{\geq}(\Gamma,d)$, $t=\text{frac}(-\sigma)$ and $f=\text{frac}(h_d)$. By the lower bound we have $t>1/2-d/2\sigma$, so $f=2t+d/\sigma -1$ and Equation \eqref{eq:t1quad} applies with $a=0$.
	
	We calculate the discriminant $\Delta$ of the quadratic
	part of Equation (\ref{eq:t1quad}). This gives us $\Delta=16\sigma^{2}+16\sigma-8d+9$.
	Using type I parameter conditions we reduce this to $\Delta=4v-8d+5$.
	As $\sigma\leq-1$ and $t\in(1/2,1)$ we have $\left[x_{h_d}\right]=d-\sigma-t>0$. So
	if $t$ is less than the smallest zero of the quadratic part of Equation \eqref{eq:t1quad}, we have proven
	that $R_{\Gamma}(\left[ x_{h_d}\right],\lfloor h_d \rfloor,d)<0$. But the smallest
	zero is precisely the assumed upper bound, seen by direct calculation.
\end{proof}

We will also consider the $x$ coordinate $\left[x_{h_d}\right]-1=d-\sigma-t-1$. Now we will analyse the rap around the point $$(\left[x_{h_d}\right]-1,\lfloor h_d\rfloor,d)=(d-\sigma-t-1,{h_d}-f,d).$$ In this case, we observe that $\lfloor h_d \rfloor=2(d-\sigma-t-1)-a+2$.
\begin{lemma}
	\label{lem:t1uppoly2}Let $\Gamma$  be in $\text{SRG}(v,k,\lambda,\mu)$
	and of type I with restricted eigenvalues $\rho>\sigma$. Then 
	\begin{equation}\label{eq:t1quad2}
	\begin{aligned}
	R_{\Gamma}(d-\sigma-t-1,&2(d-\sigma-t-1)-a+2,d)=\\&-(d-\sigma-t-1)(2t^{2}+(3+4\sigma)t+d+\sigma)\\
	&+(a-2)(t^2 + (2\sigma+1)t+a\sigma(\sigma+1)+d).
	\end{aligned}
	\end{equation}
\end{lemma}
\begin{proof}
	This can be proven by manipulating the expression $R(z,2z-a+2)$, where $z=(d-\sigma-t-1)$, into two parts, one containing only parts divisible by $a-2$. We can also substitute values into Equation \eqref{eq:t1quad} from Lemma \ref{lem:t1uppoly} directly.
\end{proof}

Unfortunately the quadratic found in the first part of Equation \eqref{eq:t1quad2} does not have positive smallest root most of the time. However, we still get sufficient conditions which may be applicable in situations when $d$ is large, relative to $v$.

\begin{prop}
	\label{prop:strictgen2}Let $\Gamma$ be in $\text{SRG}(v,k,\lambda,\mu)$
	and of type I with restricted eigenvalues $\rho>\sigma$ with $\sigma<-2$, and $d$ be a non-negative integer where $d<-\sigma$. If 
	\[
	frac(-\sigma)<\text{min}\left(\frac{d}{\sqrt{v}+1},-\frac{1}{4}+\frac{\sqrt{v}-\sqrt{v-2d+5/4}}{2}\right),
	\]
	then $\text{Rab}_{\geq}(\Gamma,d)< \lfloor\text{Haem}_{\geq}(\Gamma,d) \rfloor$.
\end{prop}
\begin{proof}
	Let $h_d=\text{Haem}_{\geq}(\Gamma,d)$, $t=\text{frac}(-\sigma)$ and $f=\text{frac}(h_d)$. As $t<-d/2\sigma$ we have  $f=2t+d/\sigma+1$, and Equation \eqref{eq:t1quad2} applies with $a=2$.  
	
	We calculate the discriminant $\Delta$ of the quadratic
	part of Equation (\ref{eq:t1quad2}). This gives us $\Delta=16\sigma^{2}+16\sigma-8d+9$.
	Using type I parameter conditions we reduce this to $\Delta=4v-8d+5$.
	As $\sigma<-2$ and $t\in(0,1/2)$ we have $\left[x_{h_d}\right]-1=d-\sigma-t-1>0$. So
	if $t$ is less than the smallest zero of the quadratic part of Equation \eqref{eq:t1quad2}, we have proven
	that $R_{\Gamma}(\left[ x_{h_d}\right]-1,\lfloor 	h_d\rfloor,d)<0$. But the smallest
	zero is precisely the second assumed upper bound, seen by direct calculation.
\end{proof}

\begin{prop}
	\label{prop:strictgen3}Let $\Gamma$ be in $\text{SRG}(v,k,\lambda,\mu)$
	and of type I with restricted eigenvalues $\rho>\sigma$ with $\sigma<-3$, and $d$ be a non-negative integer where $d<-\sigma$. If 
	\[
	\frac{d}{\sqrt{v}+1}<frac(-\sigma)<\text{min}\left(\frac{1}{2}+\frac{d}{\sqrt{v}+1},-\frac{1}{4}+\frac{\sqrt{v}-\sqrt{v-2d+5/4}}{2}\right),
	\]
	then $\text{Rab}_{\geq}(\Gamma,d)< \lfloor\text{Haem}_{\geq}(\Gamma,d) \rfloor$.
\end{prop}
\begin{proof}
	Let $h_d=\text{Haem}_{\geq}(\Gamma,d)$, $t=\text{frac}(-\sigma)$ and $f=\text{frac}(h_d)$. As $-d/2\sigma<t<1/2 - d/2\sigma$ we have  $f=2t+d/\sigma$, and Equation \eqref{eq:t1quad2} applies with $a=1$.
	
	Consider the part of Equation \eqref{eq:t1quad2} divisible by $a-2$. This quadratic is larger than $\sigma(\sigma+1)+2\sigma+1>0$ as $\sigma<-3$. Then we observe that $a-2$ is negative, so this quadratic contributes negatively to the value $R_{\Gamma}(\left[x_{h_d}\right],\lfloor h_d \rfloor, d)$. 
	
	The result follows similarly to Proposition \ref{prop:strictgen2}.  
	\end{proof}

\subsection*{Improvements for infinitely many type I graphs}
We now introduce a family of type I strongly regular graphs, and prove that for any $d$, $\text{Rab}_{\geq}(\Gamma,d)<\lfloor \text{Haem}_\geq(\Gamma,d)\rfloor$ for infinitely many graphs from this family. The graphs we consider are a well-known family of strongly regular graphs, and are used in Greaves and Soicher \cite{GS_2016} for a similar purpose as our own. 

Let $q$ be a power of a prime, with $q\equiv1\ (\text{mod }4)$.
Then the \emph{Paley graph} of order $q$, denoted by $P_q$,  has vertex set $V=\mathbb{F}_{q}$,
with two vertices adjacent if and only if their difference is a square
in $\mathbb{F}_{q}^{*}$. Paley graphs are an example of an infinite family of type I strongly regular graphs (see Godsil \cite{G_1993}), and the Paley graph of order $q$ belongs to $\text{SRG}(q,(q-1)/2,(q-5)/4,(q-1)/4)$.

\begin{thm}\label{thm:infpaley}
	Let $d$ be a non-negative integer. Then there are infinitely many primes $p$ such that $p\equiv1\ (\text{mod }4)$ and $\text{Rab}_{\geq}(P_p,d)<\lfloor\text{Haem}_\geq(P_p,d)\rfloor.$
\end{thm}
\begin{proof}
	Let $a,b$ lie in the interval $(0,1/4)$, and $a<b$. First we note that for  \\ $\mathcal{P}_{1}=\{p\ \text{prime};p\equiv1\ (\text{mod }4)\}$, the set $\{\sqrt{p}/2;p\in\mathcal{P}_{1}\}$ is uniformly distributed modulo $1$. We can see this result as a direct application of Yip \cite[Corollary 6.3]{Y_2021}, or by a slight adjustment of the arguments in Shubin \cite{S_2020}. In \cite{GS_2016}, a paper of Balog \cite{B_1985} is cited to contain this result, although no direct statement is given for this specific case.
	
	In particular, we know that there exists infinitely many primes $p$ in $\mathcal{P}_1$ such that $\text{frac}(\sqrt{p}/2)\in(a,b)$, and $d<(\sqrt{p}+1)/2$. Denote this set of primes by $\mathcal{Q}$.
    
    For $p\in \mathcal{Q}$, consider the values 
    \begin{equation*}
    \frac{d}{\sqrt{p}+1}\text{ and } \frac{1}{4}+\frac{\sqrt{p}-\sqrt{p-2d+5/4}}{2}.
    \end{equation*}
    The first value tends to $0$ and the second value tends to $1/4$ as $p$ goes to infinity, so eventually 
    \begin{equation*}
    \frac{d}{\sqrt{p}+1}<a<b< \frac{1}{4}+\frac{\sqrt{p}-\sqrt{p-2d+5/4}}{2}.
    \end{equation*}
    We know $\text{frac}(\sqrt{p}/2)\in (0,1/4)$, so $\text{frac}(\sqrt{p}/2+1/2)=\text{frac}(\sqrt{p}/2)+1/2$, and 
    \begin{equation*}
		\frac{1}{2}+\frac{d}{\sqrt{p}+1}<\text{frac}(\sqrt{p}/2+1/2)< \frac{3}{4}+\frac{\sqrt{p}-\sqrt{p-2d+5/4}}{2}.
	\end{equation*}
    For such a prime $p$ the Paley graph $P_p$ is a strongly regular graph satisfying the conditions of Proposition \ref{prop:strictgen}. Therefore, there exists infinitely many primes $p\in \mathcal{Q}$ such that $P_p$ satisfies the conditions of Proposition \ref{prop:strictgen}, which gives the result.  
\end{proof}

As Paley graphs are type I strongly regular graphs, the above gives us the following Corollary.

\begin{cor}
	Let $d$ be a non-negative integer. Then there are infinitely many type I strongly regular graphs $\Gamma$ for which $\text{Rab}_{\geq}(\Gamma,d)<\lfloor\text{Haem}_\geq(\Gamma,d)\rfloor.$
\end{cor}

\subsection{Type II strongly regular graphs}\label{sec:tIIstrict}

 Throughout this Section, we will assume $\Gamma$ is a type II strongly regular graph with parameters $(v,k,\lambda,\mu)$ and restricted eigenvalues $\rho>\sigma$. Further, we will let $d$ be a non-negative integer and $h_d=Haem_{\geq}(\Gamma,d)$.

In this case, we have $x_{h_d}=d-\sigma-1/2$ and
\begin{equation*}
  \left[x_{h_d}\right]=d-\sigma-1,
\end{equation*}
and we only have one unknown fractional part, $\text{frac}(h_d)$.

Let $d<y<v$ and $0<f<y$. We have 
\begin{equation}
  x_y-x_{(y-f)}=v(k-d)\frac{f}{(v-y)(v-y+f)}.
\end{equation}
At the start of this section, we commented that it would be useful if we knew that $\left[ x_{\lfloor h_d \rfloor}\right]=\left[x_{h_d}\right]$. For type II graphs, we know when this is true.
\begin{lemma}\label{lem:xeqtII} 
	Let $\Gamma$ be in $\text{SRG}(v,k,\lambda,\mu)$
	and of type II with restricted eigenvalues $\rho>\sigma$, and let $d$ be an integer, $0\leq d \leq k$. 
	
	Then
	$\left[x_{\lfloor h_d \rfloor}\right]=\left[x_{h_d}\right]$ if and only if
	\begin{equation*}
		\text{frac}(h_d)<\frac{v(k-d)}{(k-\sigma)(k-\sigma-1)}.
	\end{equation*}
	
\end{lemma}
\begin{proof}
	Let $f=\text{frac}(h_{d}),\delta_f=x_{h_{d}}-x_{(h_{d}-f)}$. Then $\left[x_{\lfloor h_{d} \rfloor}\right]=\left[x_{h_d}\right]$ if and only if $0\leq\delta_f<1$. This is equivalent to 
	\begin{eqnarray*}
	f&<&\frac{(v-h_d)^2}{v(k-d)-(v-h_d)}\\
	 &=&\frac{v(k-d)}{(k-\sigma-1)(k-\sigma)}.
	\end{eqnarray*}
\end{proof}
The next corollary shows that for $d$ small enough, this condition always holds. 
\begin{cor}\label{cor:d0xeqtII}
	Let $\Gamma$ be in $\text{SRG}(v,k,\lambda,\mu)$
	and of type II with restricted eigenvalues $\rho>\sigma$, and let $d$ be an integer, $0\leq d \leq k$. 
	
	Then $\left[x_{\lfloor h_d \rfloor}\right]=\left[x_{h_d}\right]$ if $d \leq k-(k-\sigma)(k-\sigma-1)/v.$
\end{cor}
\begin{proof}
	If $d \leq k-(k-\sigma)(k-\sigma-1)/v$, then $v(k-d)/(k-\sigma-1)(k-\sigma)\geq1$. As $f<1$, by Lemma \ref{lem:xeqtII} we are done.
\end{proof}
Particular attention is given to the case of independent sets ($0$-regular induced subgraphs) in the literature. The next corollary shows that for strongly regular graphs which are not complete multipartite, we know that $\left[ x_{\lfloor h_0 \rfloor}\right]=\left[ x_{h_0}\right]$.
\begin{cor}\label{cor:xd0tII}
	Let $\Gamma$ be in $\text{SRG}(v,k,\lambda,\mu)$ and of type II with restricted eigenvalues $\rho>\sigma$, and let $d$ be an integer, $0\leq d \leq k$. 

	If $\mu<k$, then $\left[ x_{\lfloor h_0 \rfloor}\right]=\left[ x_{h_0}\right]$.
\end{cor}
\begin{proof}
	If $\mu<k$, we must have $\rho>0$. Using Proposition \ref{prop:params}, we see that 
	\begin{eqnarray*}
	k(k-\rho)-\mu(k-\sigma-1)&=&k(k-\mu-\rho)+\mu(\sigma+1)\\
	&=&k(-\rho\sigma-\rho)+\mu(\sigma+1)\\
	&=&(\sigma+1)(\mu -\rho k)\geq 0. 
	\end{eqnarray*}
    From here we use Lemma \ref{prop:vmu} to see that
    \begin{eqnarray*}
    \frac{(k-\sigma)(k-\sigma-1)}{v}&=&\frac{\mu(k-\sigma-1)}{(k-\rho)}\\
    &\leq& k.
    \end{eqnarray*}
    We now see that Corollary \ref{cor:d0xeqtII} applies to $d=0$.
\end{proof}

Now we know we can have $\left[ x_{\lfloor h_d \rfloor}\right]=\left[ x_{h_d}\right]$, so we are interested in the value of $R_{\Gamma}$ at $x$ coordinate $\left[x_{h_d}\right]=d-\sigma-1$. We already know a root of this polynomial in its second coordinate, so we have the following result.

\begin{lemma}\label{lem:tIIRneg}
	Let $\Gamma$ be in $\text{SRG}(v,k,\lambda,\mu)$ with restricted eigenvalues $\rho>\sigma$, and let $d$ be an integer, $0\leq d \leq k$. 
	
	Then $R_{\Gamma}(d-\sigma-1,y,d)<0$ if and only if 
	\begin{equation*}
		\frac{(k-\sigma)(d-\sigma-1)}{\mu}<y<h_d.
	\end{equation*}
\end{lemma} 
\begin{proof}
	In Lemma \ref{lem:upbound}, we show that $R_\Gamma(x,h_d,d)$ is a quadratic in $x$ with positive leading coefficient, critical point $x_y$ and roots $x_y+1/2=d-\sigma$ and $x_y-1/2=d-\sigma-1$. Now regarding $R_{\Gamma}(d-\sigma-1,y,d)$ as a quadratic in $y$, we see that $R_{\Gamma}$ has leading coefficient $\mu$ and one root at $h_d$. The constant term of this polynomial is $(d-\sigma)(d-\sigma-1)v$, so the other root of the polynomial is
	\begin{eqnarray*}
	\alpha &=& v(d-\sigma)(d-\sigma-1)/\mu h_d \\
	&=& (k-\sigma)(d-\sigma-1)/\mu
	\end{eqnarray*}  
\end{proof}

Now we can derive a simple condition for the raub to be strictly better than Haemers' upper bound, $h_d$, for type II strongly regular graphs.

\begin{cor}\label{cor:tIIstr}
	Let $\Gamma$ be in $\text{SRG}(v,k,\lambda,\mu)$ and of type II with restricted eigenvalues $\rho>\sigma$, and let $d$ be an integer, $0\leq d \leq k$. 
	
	Suppose
	\begin{equation*}
		0<\text{frac}(h_d)<\frac{(k-\sigma)-(d-\sigma)(\rho-\sigma)}{\mu}.
	\end{equation*}
	Then $\text{Rab}_{\geq}(\Gamma,d)<\lfloor h_d \rfloor$.
\end{cor}
\begin{proof}
	By Lemma \ref{lem:tIIRneg}, $R_{\Gamma}(d-\sigma-1,\lfloor h_d \rfloor, d)<0$ if and only if
	\begin{eqnarray*}
	0<\text{frac}(h_d)&<&h_d-(k-\sigma)(d-\sigma-1)/\mu\\
	&=& \frac{(k-\sigma)-(d-\sigma)(\rho-\sigma)}{\mu}
	\end{eqnarray*} 
\end{proof}

\subsection*{Improvements for infinitely many type II graphs}\label{sec:CY1strict}

We now introduce a family of type II strongly regular graphs, and prove that for any $d$, $\text{Rab}_{\geq}(\Gamma,d)<\lfloor \text{Haem}_\geq(\Gamma,d)\rfloor$ for infinitely many graphs $\Gamma$ from this family.

The graphs we consider are presented in Calderbank and Kantor \cite[Example CY1]{CK_1986}. For each $q$ a prime power and $l,i$ positive integers where $i\leq q$, they find a strongly regular graph using codes over the finite fields. 
We will be interested in the case where $l=2,i=q-1$. Then by \cite[Table 2a]{CK_1986}, there is a strongly regular graph $\Gamma_q$ with parameters $(v,k,\lambda,\mu)$ and restricted eigenvalues $\rho>\sigma$ as follows.
\begin{eqnarray}\label{eq:gqparms}
v&=&q^4,\nonumber\\
k&=&(q-1)^2(q^2+1),\nonumber\\
\lambda&=&\mu+\rho+\sigma,\\
\mu&=&(q-1)^2((q-1)^2+1),\nonumber\\
\rho&=&(q-1)^2,\nonumber\\
\sigma&=&1-2q.\nonumber
\end{eqnarray}    

\begin{thm}
	Let $\Lambda>0$. Then there are infinitely many prime powers $q$ such that for each integer $d$ in the range $0\leq d\leq \Lambda$, we have
	\begin{equation*}
	\text{Rab}_{\geq}(\Gamma_q,d)<\lfloor\text{Haem}_\geq(\Gamma_q,d)\rfloor.
	\end{equation*}
\end{thm}
\begin{proof}
	Consider graphs $\Gamma_q$ with parameters from Equation \eqref{eq:gqparms}. For large enough $q$ and fixed $d<\Lambda$, we can calculate the fractional part of 
	$h_d =\text{Haem}_\geq(\Gamma,d)$, and show that it tends to $0$, but is never equal to $0$. On the other hand, the upper bound in Corollary \ref{cor:tIIstr} can also be calculated, and shown to tend to $1$.
	 
	Therefore, for $q$ large enough we can apply Corollary \ref{cor:tIIstr} to get the result.
\end{proof}

As the graphs $\Gamma_q$ are type II strongly regular graphs, the above gives us the following Corollary.

\begin{cor}
	Let $d$ be a non-negative integer. Then there are infinitely many type II strongly regular graphs $\Gamma$ for which $\text{Rab}_{\geq}(\Gamma,d)<\lfloor\text{Haem}_\geq(\Gamma,d)\rfloor.$
\end{cor}

\subsection{Computational comparison}\label{sec:rabcompcomp}

Now we investigate when the raub and ralb are strictly better than the bounds of Haemers. For this, we can use the \textsf{AGT} package for \textsf{GAP} \cite{AGT_2020}. We start by explaining how we compute the regular adjacency bounds.

\subsubsection*{A note on computing the regular adjacency bounds}\label{sec:comprabs}

To compute the regular adjacency upper bound, we will iterate through values of $y$, starting at $v$ and decreasing by one at the end of each step. For each value of $y$, we determine if $R(m,y,d)\geq 0$ for all integers $m$. If we find this is not true, we have found the regular adjacency upper bound. Similarly, we can calculate the regular adjacency lower bound by iterating through values of $y$, starting from $d+1$ and increasing by one at each step.

Therefore, we would like to be able to determine whether $R(m,y,d)\geq 0$ for all integers $m$ efficiently. We have two cases for $y$:
\begin{enumerate}
	\item If $v=y$, $R$ has degree at most 1 in $x$, and is non-negative at all integers if and only if $R$ is a non-negative constant function. 
	\item If $v>y$, then we only need to consider the $x$ coordinate $\left[x_y\right]$ ($x_y$ defined in Equation \eqref{eq:x_y}) by the statement \eqref{eq:R_gefficient}.
\end{enumerate}
The most computationally costly calculation in this method is the rounding of $x_y$ to $\left[x_y\right]$. As $x_y$ is a rational number, the rounding operation can be done exactly and efficiently.

\subsubsection*{Comparison using the \textsf{AGT} package}

Let $(v,k,\lambda,\mu)$ be feasible strongly regular graph parameters (see \cite{AGT_2020} for the definition of feasible parameters). For any graph $\Gamma\in \text{SRG}(v,k,\lambda,\mu)$ and any non-negative integer $d\leq k$, we can compare the bounds of Haemers and our new bounds by using the following functions found in the \textsf{AGT} package:
\begin{itemize}
	\item \verb|HaemersRegularUpperBound| to calculate $\text{Haem}_{\geq}(\Gamma,d)$.
	\item \verb|HaemersRegularLowerBound| to calculate $\text{Haem}_{\leq}(\Gamma,d)$.
	\item \verb|RegularAdjacencyUpperBound| to calculate $\text{Rab}_{\geq}(\Gamma,d)$.
	\item \verb|RegularAdjacencyLowerBound| to calculate  $\text{Rab}_{\leq}(\Gamma,d)$.
\end{itemize}  
\begin{rem}
	A graph $\Gamma \in \text{SRG}(v,k,\lambda,\mu)$ is called \emph{imprimitive} if it $\Gamma$ or $\overline{\Gamma}$ is not connected  (see \cite{BH_2011}). This is equivalent to $\mu\in \{0,k\}$, in which case $\Gamma$ is isomorphic to a union of complete graphs or the complement of a union of complete graphs. In this cases, it is easy to show that Haemers' bounds and the regular adjacency bounds coincide. For this reason, we only compare our bounds for primitive strongly regular graphs (i.e. non-imprimitive strongly regular graphs).
\end{rem}

Much of the functionality available in the \textsf{AGT} package can be used to experiment with strongly regular graphs and their parameters. For example, some of the information about primitive strongly regular graph parameter tuples collected in Brouwer's lists \cite{B_2018} is available in the variable   \verb|AGT_Brouwer_Parameters|. In particular, \verb|AGT_Brouwer_Parameters| contains every feasible parameter tuple for a primitive strongly regular graph with at most $1300$ vertices (see \cite{BH_2011} for a definition of feasible parameter tuples).

There are natural trivial bounds to consider when comparing our new bounds with Haemers' bounds. A trivial upper bound on the order of any induced subgraph of a graph of order $v$ is $v$, and a lower bound on the order of a $d$-regular induced subgraph is $d+1$. For a graph $\Gamma\in \text{SRG}(v,k,\lambda,\mu)$ and non-negative integer $d\leq k$, we define 
\begin{align*}
\text{Haem}^{*}_{\geq}(\Gamma,d)&=\text{min}(\lfloor\text{Haem}_{\geq}(\Gamma,d)\rfloor,v),\\
\text{Haem}^{*}_{\leq}(\Gamma,d)&=\text{max}(\lceil\text{Haem}_{\leq}(\Gamma,d)\rceil,d+1).
\end{align*}

We will consider each feasible primitive strongly regular parameter tuple $(v,k,\lambda,\mu)$ where $v\leq 1300$. For each integer $d$, where $0\leq d\leq k$, we will consider a graph $\Gamma\in \text{SRG}(v,k,\lambda,\mu)$, and compare 
$\text{Rab}_{\geq}(\Gamma,d)$ with  $\text{Haem}^{*}_{\geq}(\Gamma,d)$, and  $\text{Rab}_{\leq}(\Gamma,d)$ with  $\text{Haem}^{*}_{\leq}(\Gamma,d)$. There are currently 1616218 combinations of parameter tuples and $d$ satisfying all the above conditions.

\subsubsection*{The raub and Haemers' upper bound}

The results of our calculations show that $\text{Rab}_{\geq}(\Gamma,d)< \text{Haem}^{*}_{\geq}(\Gamma,d)$ in 18199 cases, and in 10931 of these cases the regular adjacency polynomial proves there is no possible order for a $d$-regular induced subgraph in $\Gamma$. Out of the remaining 7268 cases, there are 123 cases for which $\text{Rab}_{\geq}(\Gamma,d)< \text{Haem}^{*}_{\geq}(\Gamma,d)-1$. 

Table \ref{tab:raubh} has first column consisting of all values of $d$ for which we have found a parameter tuple $(v,k,\lambda,\mu)$ such that for every  $\Gamma\in\text{SRG}(v,k,\lambda,\mu)$,  $$\text{Rab}_{\geq}(\Gamma,d)< \text{Haem}^{*}_{\geq}(\Gamma,d)-1.$$
For each of these $d$, the second column gives how many parameter tuples this occurs for ($\Sigma_d$). The third column give the largest difference $\text{Haem}^{*}_{\geq}(\Gamma,d)-\text{Rab}_{\geq}(\Gamma,d)$ found for this value of $d$ ($\Omega_{d}$).

\subsubsection*{The ralb and Haemers' lower bound}

The results of our calculations show that $\text{Rab}_{\leq}(\Gamma,d)> \text{Haem}^{*}_{\leq}(\Gamma,d)$ in 15639 cases, and in 10931 of these cases the regular adjacency polynomial proves there is no possible order for a $d$-regular induced subgraph in $\Gamma$. Out of the 4708 other cases, there are 787 cases for which $\text{Rab}_{\leq}(\Gamma,d)> \text{Haem}^{*}_{\leq}(\Gamma,d)+1$. 

Table \ref{tab:ralbh} has first column consisting of all values of $d$ for which we have found a parameter tuple $(v,k,\lambda,\mu)$ such that for every  $\Gamma\in\text{SRG}(v,k,\lambda,\mu)$, $$\text{Rab}_{\leq}(\Gamma,d)> \text{Haem}^{*}_{\leq}(\Gamma,d)+1.$$
For each of these $d$, the second column gives how many parameter tuples this occurs for ($\Sigma_{d}$). The third column give the largest difference $\text{Rab}_{\leq}(\Gamma,d)-\text{Haem}^{*}_{\leq}(\Gamma,d)$ found for this value of $d$ ($\Omega_{d}$).  

\begin{table}[ht!]
	\centering
	\begin{tabular}{c|c|c}
		$d$ & $\Sigma_{d}$ & $\Omega_{d}$\tabularnewline
		\hline 
		0 & 44 & 5\tabularnewline
		\hline 
		1 & 27 & 4\tabularnewline
		\hline 
		2 & 20 & 4\tabularnewline
		\hline 
		3 & 10 & 3\tabularnewline
		\hline 
		4 & 6 & 3\tabularnewline
		\hline 
		5 & 5 & 3\tabularnewline
		\hline 
		6 & 3 & 2\tabularnewline
		\hline 
		7 & 2 & 2\tabularnewline
		\hline 
		8 & 2 & 2\tabularnewline
		\hline 
		9 & 2 & 2\tabularnewline
		\hline 
		10 & 1 & 2\tabularnewline
		\hline 
		11 & 1 & 2\tabularnewline
	\end{tabular}
	\caption{Cases for which $\text{Rab}_{\geq}(\Gamma,d)< \text{Haem}^{*}_{\geq}(\Gamma,d)-1$.}\label{tab:raubh}
\end{table}

\begin{table}[ht!]
	\begin{minipage}[t]{0.45\columnwidth}%
		\begin{center}
			\begin{tabular}{c|c|c}
				$d$ & $\Sigma_{d}$ & $\Omega_{d}$\tabularnewline
				\hline 
				2 & 24 & 120\tabularnewline
				\hline 
				3 & 41 & 94\tabularnewline
				\hline 
				4 & 62 & 69\tabularnewline
				\hline 
				5 & 64 & 45\tabularnewline
				\hline 
				6 & 85 & 66\tabularnewline
				\hline 
				7 & 79 & 43\tabularnewline
				\hline 
				8 & 69 & 33\tabularnewline
				\hline 
				9 & 71 & 16\tabularnewline
				\hline 
				10 & 62 & 16\tabularnewline
				\hline 
				11 & 61 & 18\tabularnewline
				\hline 
				12 & 34 & 17\tabularnewline
			\end{tabular}
			\par\end{center}%
	\end{minipage}\hfill{}%
	\begin{minipage}[t]{0.45\columnwidth}%
		\begin{center}
			\begin{tabular}{c|c|c}
				$d$ & $\Sigma_{d}$ & $\Omega_{d}$\tabularnewline
				\hline 
				13 & 37 & 27\tabularnewline
				\hline 
				14 & 27 & 6\tabularnewline
				\hline 
				15 & 18 & 8\tabularnewline
				\hline 
				16 & 16 & 7\tabularnewline
				\hline 
				17 & 15 & 8\tabularnewline
				\hline 
				18 & 8 & 5\tabularnewline
				\hline 
				19 & 8 & 4\tabularnewline
				\hline 
				20 & 3 & 3\tabularnewline
				\hline 
				21 & 2 & 2\tabularnewline
				\hline 
				22 & 1 & 4\tabularnewline
			\end{tabular}
			\par\end{center}%
	\end{minipage}
	\caption{Cases for which $\text{Rab}_{\leq}(\Gamma,d)> \text{Haem}^{*}_{\leq}(\Gamma,d)+1$.}\label{tab:ralbh}
\end{table}

\begin{rem}
In many cases we can further improve a bound on the order of a $d$-regular induced subgraph by considering a well-known divisibility condition, that for any $d$-regular graph of order $y$, we must have that $2$ divides $yd$.

For example, consider a graph $\Gamma$ from $\text{SRG}(41,20,9,10)$ and $d=1$. We can use the \textsf{AGT} package to find that $\text{Rab}_{\geq}(\Gamma,1)=7$. But a $1$-regular induced subgraph must have even order. Therefore, any $1$-regular induced subgraph of $\Gamma$ has order at most $6$. 
\end{rem}
\section{The Clique Adjacency Bound}\label{sec:cab}

An early application of the block intersection polynomial can be found in Soicher \cite{S_2010}. In this paper, Soicher derives a bound for the order of cliques in edge-regular graphs. Here, we present the main tools found in Soicher \cite{S_2010} and Greaves and Soicher \cite{GS_2016}, and then investigate these tools computationally.

A graph $\Gamma$ is edge-regular with parameters $(v,k,\lambda)$ if it is a $v$-vertex $k$-regular graph containing at least one edge, and such that each pair of adjacent vertices have exactly $\lambda$ common neighbours. For such a graph $\Gamma$, Soicher defines the \emph{clique adjacency polynomial}, $C_{\Gamma}(x,y)$ as
\begin{equation*}
C_{\Gamma}(x,y):=(v-y)x(x+1)-2xy(k-y+1)+y(y-1)(\lambda-y+2).
\end{equation*}

The \emph{clique adjacency bound} for $\Gamma$, denoted by $\text{CAB}(\Gamma)$, is defined to be the least integer $y\geq 2$ such that there exists an integer $m$ where $C_{\Gamma}(m,y+1)<0$. Soicher \cite{S_2010} shows that this is a bound on the order of a clique in $\Gamma$.

For strongly regular graphs (which are also edge-regular), the clique adjacency polynomial is very closely related to the regular adjacency polynomial. Let $\Gamma\in \text{SRG}(v,k,\lambda,\mu)$ with restricted eigenvalues $\rho>\sigma$. Then $\overline{\Gamma}$ is also a strongly regular graph, with the parameters $(\overline{v},\overline{k},\overline{\lambda},\overline{\mu})$ and restricted eigenvalues $\overline{\rho}>\overline{\sigma}$, where
\begin{equation*}
\begin{aligned}
\overline{v}&=v\\
\overline{k}&=v-k-1\\
\overline{\lambda}&=v-2-2k+\mu\\
\overline{\mu}&=v-2k+\lambda\\
\overline{\rho}&=-1-\sigma\\
\overline{\sigma}&=-1-\rho
\end{aligned}
\end{equation*}
(see Brouwer, Cohen and Neumaier \cite{BCN_1989}). The next result shows a relation between the regular adjacency polynomial of $\overline{\Gamma}$ and the clique adjacency polynomial of $\Gamma$.
\begin{prop}
	Let $\Gamma$ be in $\text{SRG}(v,k,\lambda,\mu)$. Then
	\begin{equation*}
	R_{\overline{\Gamma}}(x,y,0)=C_{\Gamma}(y-x-1,y).
	\end{equation*}
	In particular, we have $\text{Rab}_\geq(\overline{\Gamma},0)=\text{CAB}(\Gamma)$. 
\end{prop}
\begin{proof}
	The equality of the polynomials can be directly verified using the identities in Proposition \ref{prop:params}.  The equality  $\text{Rab}_\geq(\overline{\Gamma},0)=\text{CAB}(\Gamma)$ then follows from the definitions of the bounds.
\end{proof}
Thus for strongly regular graphs, Proposition \ref{prop:strictgen}
is a generalisation of Greaves and Soicher \cite[Theorem 1]{GS_2016}. Note that the regular adjacency polynomial is a quadratic polynomial in both $x$ and $y$, whereas the clique adjacency polynomial is
a cubic polynomial in $y$. This suggests the regular adjacency polynomial
may be easier to analyse and use in computations, when we are studying
strongly regular graphs. In particular we can show that the sufficient condition is also necessary in \cite[Theorem 2.4]{GS_2016}.
\begin{cor}
	Let $\Gamma$ be in $\text{SRG}(v,k,\lambda,\mu)$ and of type II, with restricted eigenvalues $\rho>\sigma$. Then
	$\text{CAB}(\Gamma)<\lfloor 1-k/\sigma\rfloor$ if and only if 
	\begin{equation*}
	0<frac(-k/\sigma)<1-\rho(\rho+1)/(v-2k+\lambda)
	\end{equation*}
\end{cor}
\begin{proof}
	This is an application of Corollaries \ref{cor:xd0tII} and \ref{cor:tIIstr} to  $\overline{\Gamma}$ and its parameters given above.
\end{proof}

In \cite{GS_2016}, Greaves and Soicher compare the clique adjacency bound to the well-known Delsarte bound \cite{D_1975}. They prove the clique adjacency bound is at least as good as the Delsarte bound for any strongly regular graph. They further conjecture that the clique adjacency bound is at least as good as the Hoffman ratio bound \cite{H_2021} of the complement graph for any edge-regular graph. We can use an eigenvalue bound of Abiad et al. \cite{ADDK_2020} to prove this conjecture is true. 
\begin{thm}
	Let $\Gamma$ be a non-complete edge-regular graph with parameters $(v,k,\lambda)$ and second largest eigenvalue $\rho$. Then 
	\begin{equation*}
	\text{CAB}(\Gamma)\leq \text{Haem}_{\geq}(\overline{\Gamma},0).
	\end{equation*}
\end{thm}
\begin{proof}
	First suppose $\Gamma$ is disconnected. Then $\overline{\Gamma}$ has smallest eigenvalue $-(1+k)$ and $\text{Haem}_{\geq}(\overline{\Gamma},0)=k+1$. As $\Gamma$ is non-complete, $k<\lambda+1$ and we can see that  $C_{\Gamma}(0,k+1)<0$. Therefore, $	\text{CAB}(\Gamma)\leq \text{Haem}_{\geq}(\overline{\Gamma},0).
	$ when $\Gamma$ is disconnected. 
	
	Now we will assume that $\Gamma$ is connected. For each $y<v$ we consider $x_y$, the point at which $C_{\Gamma}(x,y)$ is minimal. Then $x_y+1/2=y(k-y+1)/(v-y)$, and 
	\begin{equation*}
	C_{\Gamma}(x_y+1/2,y)=\frac{-y}{(v-y)}((v-2k+\lambda)y^2+(k^2+3k-\lambda-v(\lambda+2))y+v(\lambda+1-k)).
	\end{equation*}
	Let $s+1$ be the largest root of the quadratic factor in the  above equation. Note that the other root is negative.
	
	Using the same argument as in Lemma \ref{lem:upbound}, we see that for any $y$ such that $s+1<y<v$, there exists an integer $b_y$ such that $C_{\Gamma}(b_y,y)<0$, so we have $\text{CAB}(\Gamma)\leq s+1$. We can check that $s$ is the largest root of the quadratic in $z$, 
	\begin{equation*}
	(v-2k+\lambda)z^2+(k^2-k+\lambda-v\lambda)z-k(v-k-1),
	\end{equation*}
	by substituting $s+1$ into the quadratic above.
	
	Let $\theta_M = -k/s$. In Abiad et al. \cite{ADDK_2020} it is proven that $\rho\geq \theta_M$ in a much more general setting. Then we have
	\begin{eqnarray*}
	\text{Haem}(\overline{\Gamma},0)&=&v(1+(v-k-1)/(\rho+1))^{-1}\\
	&\geq&v(1+(v-k-1)/(\theta_M+1))^{-1}\\
	&=&s+1.
	\end{eqnarray*}
    For the details, see \cite[Theorem 2.30]{ADDK_2020}. Therefore we have $\text{CAB}(\Gamma)\leq s+1 \leq \text{Haem}(\overline{\Gamma},0)$. 
\end{proof}

Greaves and Soicher \cite{GS_2016} also remark on how tight the clique adjacency bound is for small strongly regular graphs. In particular, they ask the following question.
\begin{enumerate}[(Q)]
	\item\label{que:vk2}  Do there exist strongly regular graphs with parameters $(v,k,\lambda,\mu)$, with $k < v/2$, such that every strongly regular graph having those parameters has clique number less than the clique adjacency bound?
\end{enumerate}

It is possible to show that any parameter tuple $(v,k,\lambda,\mu)$ with the properties in question (Q) must have $v>40$. For example, we use the library of strongly regular graphs in the \textsf{AGT} package \cite{AGT_2020} in \textsf{GAP} \cite{GAP_4.11.1} to prove the following.

\begin{prop}\label{rem:vk2}There does not exist strongly regular graph parameters $(v,k,\lambda,\mu)$, with $v<41,k < v/2$, such that every strongly regular graph having those parameters has maximum clique size less than the clique adjacency bound.   
\end{prop}
\begin{proof}
	For imprimitive strongly regular graphs, it is easy to show that this is true. 
	
	The \textsf{AGT} package contains a library of primitive strongly regular graphs. For all parameters $(v,k,\lambda,\mu)$ with such parameters, we can check that the library contains a primitive strongly regular graph with maximum clique size exactly that of the clique adjacency bound. 
\end{proof}


\section*{Acknowledgements}

I would like to thank my former supervisor, Leonard Soicher, for his guidance. I would also like to thank Chi Hoi Yip and Andrei Shubin for their help with the number theory result used in Theorem \ref{thm:infpaley}. This work was supported by the Mathematical Center in Akademgorodok, under agreement No. 075-15-2019-1613 with the Ministry of Science and Higher Education of the Russian Federation.

\begin{appendices}
	\section{Verification with Maple}\label{app:maple}
	
	In this section, polynomial identities used in this paper are verified
	through the use of Maple \cite{MAPLE_2011}.
	
	We start Maple and define $R$ as the regular adjacency polynomial.
\begin{BGVerbatim}{}
!mapleprompt@>| !mapleinput@R:=x*(x+1)*(v-y)-2*x*y*k+(2*x+rho+sig+1)*y*d+y*(y-1)*mu-y*d^2;|
\end{BGVerbatim}
	To work with the parameters of strongly regular graphs, we will use
	the Maple package \verb|Groebner|. We will first define the polynomial
	ring $$P=\mathbb{Q}[t,d,v,k,lambda,mu,rho,sig],$$ where $t$ is considered as our main indeterminate.
	Then the \verb|Groebner| package will be used to calculate Gr\"obner bases and work in certain factor rings of $P$.
	
	We derive relators from Proposition \ref{prop:params} which evaluate to $0$ on
	the parameters \\$(v,k,lambda,mu)$ and restricted eigenvalues $rho,sig$ of a
	strongly regular graph.
\begin{BGVerbatim}{}
!mapleprompt@>| !mapleinput@srg_rel:={mu*(v-k-1)-k*(k-l-1),lambda-mu-rho-sig,mu-k-rho*sig};|
\end{BGVerbatim}
	For type I graphs, we add relators derived from the definition of their parameters. 
\begin{BGVerbatim}{}
!mapleprompt@>| !mapleinput@t1_rel:=srg_rel union {2*k-v+1,4*l-v+5,4*mu-v+1};|
\end{BGVerbatim}
	Now we define a monomial order for the polynomial ring $P$.
\begin{BGVerbatim}{}
!mapleprompt@>| !mapleinput@ord:=tdeg(t,d,v,k,lambda,mu,rho,sig);|
\end{BGVerbatim}
	Next we find the Gr\"obner bases of the ideals generated by the above
	relators. For more information on Gr\"obner bases, see Adams and Loustaunau \cite{AL_1994}.
\begin{BGVerbatim}{}
!mapleprompt@>| !mapleinput@G:=Groebner[Basis](srg_rel,ord);|
!mapleprompt@>| !mapleinput@H:=Groebner[Basis](t1_rel,ord);|
\end{BGVerbatim}
	Let $I$ be the ideal generated by the strongly regular graph relators \verb|srg_rel|,
	and $J$ be the ideal generated by the type 1 relators \verb|t1_rel|. The Gr\"obner
	package can then be used to do calculations in $P$ modulo the ideals $I$ and $J$. 
	
	Now we verify some of the results found in this paper using Maple.
	Some by-hand proofs have been provided in the text, but we provide a proof using Maple to further check their correctness .
	
	First we check Proposition \ref{prop:vmu}.
\begin{BGVerbatim}{}
!mapleprompt@>| !mapleinput@Groebner[NormalForm](v*mu-(k-sig)*(k-rho),G,ord);|
0 
\end{BGVerbatim}
	The next three identities are used in the results of Section \ref{sec:compb}.
	We verify the first equation from Lemma \ref{lem:upbound}.
\begin{BGVerbatim}{}
!mapleprompt@>| !mapleinput@Groebner[NormalForm](-(v-y)*factor(eval(|
!mapleprompt@>| !mapleinput@R,[x=((k-d+1)*y-v)/(v-y)]))/y|
!mapleprompt@>| !mapleinput@-(mu*y^2-((d-rho)*(k-sig)+(k-rho)*(d-sig))*y+v*(d-sig)*(d-rho)),G,ord);|
0  
\end{BGVerbatim}
	Next we verify the factorisation used later in Lemma \ref{lem:upbound}.
\begin{BGVerbatim}{}
!mapleprompt@>| !mapleinput@Groebner[NormalForm](-(v-y)*mu*factor(eval(|
!mapleprompt@>| !mapleinput@R,[x=((k-d+1)*y-v)/(v-y)]))/y|
!mapleprompt@>| !mapleinput@ -((mu*y-(d-rho)*(k-sig))*(mu*y-(d-sig)*(k-rho))),G,ord);|
0 
\end{BGVerbatim}
	The following identity is used in Lemma \ref{lem:vbound}.
\begin{BGVerbatim}{}
!mapleprompt@>| !mapleinput@Groebner[NormalForm](eval(R,[y=v,d=k]),G,ord);|
0 
\end{BGVerbatim}
	We now move to Section \ref{sec:tIstrict}. We start by checking the correctness of Proposition \ref{prop:t1bound}.
\begin{BGVerbatim}{}
!mapleprompt@>| !mapleinput@Groebner[NormalForm](expand(|
!mapleprompt@>| !mapleinput@s*(v*(d-sig)-(k-sig)*(d-sig)*(2+1/sig))),H,ord);|
0 
\end{BGVerbatim}
Finally we check Lemma \ref{lem:t1uppoly}.
\begin{BGVerbatim}{}
!mapleprompt@>| !mapleinput@Groebner[NormalForm](eval(R,[x=d-sig-t,y=2*d-2*sig-2*t])|
!mapleprompt@>|  !mapleinput@+(d-sig-t)*(2*t^2-(1-4*sig)*t+d-3*sig-1),H,ord);|
0 
\end{BGVerbatim}
We note that other results found in Section \ref{sec:strict} can also be checked in a similar way, although we do not give the full details here.

\end{appendices}

\end{document}